\documentclass[a4paper,11pt,fleqn]{article}
\usepackage{pstricks}
\usepackage{amsmath}
\usepackage{amssymb}
\usepackage{theorem} 
\usepackage{booktabs}
\usepackage{euscript}
\usepackage{exscale,relsize}
\usepackage{graphicx}
\usepackage{booktabs}
\usepackage{srcltx}
\usepackage{xcolor}
\usepackage{charter}
\usepackage{url}
\newcommand{\email}[1]{\href{mailto:#1}{\nolinkurl{#1}}}
\PassOptionsToPackage{normalem}{ulem}
\usepackage{ulem}

\renewcommand{\leq}{\ensuremath{\leqslant}}
\renewcommand{\geq}{\ensuremath{\geqslant}}
\newcommand{\minimize}[2]{\ensuremath{\underset{\substack{{#1}}}%
{\text{\rm minimize}}\;\;#2 }}

\newcommand{\pair}[2]{\langle{{#1},{#2}}\rangle} 
 
\newcommand{\scal}[2]{{\left\langle{{#1}\mid{#2}}\right\rangle}}

\newcommand{\menge}[2]{\big\{{#1}~\big |~{#2}\big\}} 
\newcommand{\Menge}[2]{\Big\{{#1}~\Big |~{#2}\Big\}} 
\newcommand{\MEnge}[2]{\Bigg\{{#1}~\Big |~{#2}\Bigg\}}

\newcommand{\BL}{\ensuremath{\EuScript B}\,}
\newcommand{\RX}{\ensuremath{\left]-\infty,+\infty\right]}}
\newcommand{\RXX}{\ensuremath{\left[-\infty,+\infty\right]}}
\newcommand{\HH}{\ensuremath{{\mathcal H}}}
\newcommand{\HHH}{\ensuremath{\boldsymbol{\mathcal H}}}

\newcommand{\XX}{\ensuremath{{\mathcal X}}}

\newcommand{\YY}{\ensuremath{\mathcal{Y}}}
\newcommand{\Sum}{\ensuremath{\displaystyle\sum}}
\newcommand{\emp}{\ensuremath{{\varnothing}}}
\newcommand{\prox}{\ensuremath{\text{\rm prox}}}

\newcommand{\Id}{\ensuremath{\operatorname{Id}}\,}

\newcommand{\RR}{\ensuremath{\mathbb{R}}}

\newcommand{\RP}{\ensuremath{\left[0,+\infty\right[}}

\newcommand{\RPP}{\ensuremath{\left]0,+\infty\right[}}
\newcommand{\NN}{\ensuremath{\mathbb N}}
\newcommand{\KK}{\ensuremath{\mathcal Z}}

\newcommand{\nn}{\ensuremath{\mathsf{n}}}
\newcommand{\pinf}{\ensuremath{{+\infty}}}

\newcommand{\dom}{\ensuremath{\text{\rm dom}\,}}

\newcommand{\card}{\ensuremath{\text{\rm card}\,}}

\newcommand{\zeroun}{\ensuremath{\left]0,1\right[}}  
\newcommand{\rzeroun}{\ensuremath{\left]0,1\right]}}

\newcommand{\ran}{\ensuremath{\text{\rm ran}\,}}

\newcommand{\gra}{\ensuremath{\text{\rm gra}\,}}

\newcommand{\conv}{\ensuremath{\operatorname{conv}}}

\newcommand{\trans}{^{\scriptscriptstyle\top}}

\newcommand{\diag}{\operatorname{diag}}
\newcommand{\W}{{\boldsymbol{Y}}}

\newtheorem{theorem}{Theorem}[section]

\newtheorem{proposition}[theorem]{Proposition}
\newtheorem{assumption}[theorem]{Assumption}
\theoremstyle{plain}{\theorembodyfont{\rmfamily}%
}
\theoremstyle{plain}{\theorembodyfont{\rmfamily}%
\newtheorem{example}[theorem]{Example}}
\theoremstyle{plain}{\theorembodyfont{\rmfamily}%
\newtheorem{remark}[theorem]{Remark}}
\theoremstyle{plain}{\theorembodyfont{\rmfamily}%
}
\theoremstyle{plain}{\theorembodyfont{\rmfamily}%
}
\theoremstyle{plain}{\theorembodyfont{\rmfamily}%
}
\theoremstyle{plain}{\theorembodyfont{\rmfamily}
\newtheorem{fact}[theorem]{Fact}}
\theoremstyle{plain}{\theorembodyfont{\rmfamily}
}
\theoremstyle{plain}{\theorembodyfont{\rmfamily}
\newtheorem{notation}[theorem]{Notation}}

\numberwithin{equation}{section}
\definecolor{labelkey}{rgb}{0,0.08,0.45}
\definecolor{refkey}{rgb}{0,0.6,0.0}
\definecolor{Brown}{rgb}{0.45,0.0,0.05}
\definecolor{dgreen}{rgb}{0.00,0.49,0.00}
\definecolor{dblue}{rgb}{0,0.08,0.75}
\RequirePackage[dvips,colorlinks,hyperindex]{hyperref}
\hypersetup{linktocpage=true,citecolor=dblue,linkcolor=dgreen}
\title{\vskip -12mm\sffamily\LARGE 
Learning with Optimal Interpolation Norms\thanks{The work of 
P. L. Combettes was
partially supported by the National Science Foundation under grant 
CCF-1715671, the work of C. A. Micchelli was supported by the
National Science Foundation under grant DMS-1522339, and the work
of M. Pontil was partially supported by the Engineering and
Physical Science Council grant EP/P009069/1.}}

\author{Patrick L. Combettes \\
\small North Carolina State University,
Department of Mathematics\\
\small Raleigh, NC 27695-8205, USA\\
\small \email{plc@math.ncsu.edu} \\[3mm] 
Andrew M. McDonald \\
\small University College London, Department of Computer Science,
London WC1E 6BT, UK\\
\small \email{a.mcdonald@cs.ucl.ac.uk}\\[3mm]
Charles A. Micchelli\\
\small State University of New York,
The University at Albany\\
\small Department of Mathematics and Statistics,
Albany, NY 12222, USA\\
\small \email{charles\_micchelli@hotmail.com} \\[3mm]
Massimiliano Pontil\\
\small Istituto Italiano di Tecnologia,
16163 Genoa, Italy\\
\small \email{massimiliano.pontil@iit.it} \\
\small and\\
\small University College London, Department of Computer Science,
London WC1E 6BT, UK 
}

\date{\ttfamily ~}
\begin{document}

\maketitle
\begin{abstract} 
\noindent 
We analyze a class of norms defined via an optimal interpolation
problem involving the composition of norms and a linear operator.
This construction, known as infimal postcomposition in convex
analysis, is shown to encompass various of norms which have been
used as regularizers in machine learning, signal processing, and
statistics. In particular, these include the latent group lasso,
the overlapping group lasso, and certain norms used for learning
tensors. We establish basic properties of this class of norms and
we provide dual norms. The extension to more general classes of
convex functions is also discussed. A stochastic block-coordinate 
version of the Douglas-Rachford algorithm is devised
to solve minimization problems involving these
regularizers. A prominent feature of the algorithm is that it yields
iterates that converge to a solution in the case of non smooth
losses and random block updates. Finally, we present numerical
experiments with problems employing the latent group lasso penalty.
\end{abstract} 

\newpage

\section{Introduction}
\label{sec:intro}

In various areas of data analysis such as machine learning,
statistics, and signal processing, regularization is a standard
tool used to promote known structures in the solutions of an
optimization problem. Structured sparsity
regularizers have been of particular interest, as a means to
encourage specific sparsity patterns in regression vectors, spectra
of matrices, gradient of images, and signal decompositions.
Important examples include the group lasso \cite{Yuan2006} and 
related norms \cite{Jacob2009-GL,Micc13,Zhao09}, 
spectral regularizers for low rank matrix learning and multitask 
learning \cite{Argyriou2008,Jaggi2010,Srebro2005}, 
multiple kernel learning \cite{Bach2004,Micc07}, 
regularizers for learning tensors \cite{Romera-Paredes2013,Tomi13}. 
Structured sparsity regularizers also arise in 
speech processing \cite{Asae14}, 
high-dimensional inverse problems \cite{Bour14}, 
image decomposition \cite{Liux15},
adaptive image interpolation \cite{Mall10},
face reconstruction \cite{Suny15},
and hyperspectral imaging \cite{Zhan16}.

We introduce a general formulation which captures the above 
norms and allows us to construct new ones via an optimal
interpolation problem involving simpler norms and a linear operator. 
Specifically, given two Banach spaces $\XX$ and $\YY$,
a norm $\|\cdot\|$ on $\YY$ is constructed as
\begin{equation}
\label{e:-1}
(\forall y\in\YY)\quad
\|y\|=\inf_{\substack{x\in \XX\\ Lx=y}}|||F(x)|||,
\end{equation}
where $F\colon\XX\to\RR^m$ is a mapping the components
of which are norms, $|||\cdot|||$ is a norm on $\RR^m$, and 
$L\colon\XX\to\YY$ is a linear operator.
As we shall see, this concise formulation encompasses many classes
of regularizers, either via the norm \eqref{e:-1} 
or the associated dual norm.
In particular, in machine learning, functions of the type 
described in \eqref{e:-1} have been investigated in 
\cite{AMP,Jacob2009-GL,Mich05,Maur12,Tomi13}. This formulation
also arises in image recovery \cite{Jnca14}, in inverse problems
\cite{Chan12}, and in the theory of interpolation spaces 
\cite{Peet70,Trie78}. 

We provide basic 
properties (cf. Proposition~\ref{p:milanojfY6Tr-21}
and Theorem~\ref{p:milanojfY6Tr-22}) and examples of this
construction. These include the overlapping group
lasso, the latent group lasso, and various norms used in tensor
learning problems. We also consider the more general formulation 
\begin{equation} 
\label{e:jfiqo-24}
\varphi(y)=\inf_{\substack{x\in\XX\\ Lx=y}}h\big(F(x)\big),
\end{equation} 
where $h$ and the components of $F$ are convex
functions that satisfy certain properties.
Such constructs are found for instance in 
signal recovery formulations \cite{Jnca14}. In convex analysis,
they are known as infimal postcomposition \cite{Livre1} and 
their importance was first underlined in \cite{Rock70}.

On the numerical side, we shall take advantage of the fact
that the mapping $F$ and the operator $L$ can be composed 
of a large number of ``simple'' components to create complex 
structures. Specifically, we present a general method to solve
optimization problems
involving regularizers of the form \eqref{e:-1} where
$|||\cdot|||=|\cdot|_1$. This method is based on 
the stochastic block-coordinate Douglas-Rachford 
iterative framework of \cite{Siop15,MaPr17}. 
Unlike existing methods, this approach guarantees convergence of
the iterates to a solution 
even when none of the functions present in the model is
differentiable and when random coordinates updates are performed. 

The paper is organized as follows. 
In Section~\ref{sec:norms-definition} we
introduce the general class of regularizers and establish some of
their basic properties. In Section~\ref{sec:examples} we construct a
number of examples of norms within the proposed framework.
In Section~\ref{sec:opt} we present a random block-coordinate
algorithm to solve learning problems involving these regularizers.
Finally, in Section \ref{sec:exps}, we report on numerical
experiments with this algorithm using the latent group lasso
penalty.

{\bfseries Notation.}
We introduce our notation and recall basic concepts from convex 
analysis used throughout the paper; for details see 
\cite{Livre1,Zali02}. Let $\XX$ be a real Banach space, let
$\|\cdot\|$ be its norm, let $\XX^*$ be its topological dual, and
let $\pair{\cdot}{\cdot}$ be the canonical bilinear form on
$\XX\times\XX^*$. If $\XX\neq\{0\}$ and $\XX$ is reflexive, it
follows from James' theorem that the norm of $\XX^*$ is defined by
\begin{equation} 
\label{e:dualnorm}
(\forall x^*\in\XX^*)\quad
\|x^*\|_*=\max_{\substack{x\in\XX\\ \|x\|=1}}\pair{x}{x^*}.
\end{equation}
The space of bounded linear operators from $\XX$ to a Banach space
$\YY$ is denoted by $\BL(\XX,\YY)$.
Let $\varphi\colon\XX\to\RX$. The domain of $\varphi$ is 
$\dom\varphi=\menge{x\in\XX}{\varphi(x)<\pinf}$ and the conjugate 
of $\varphi$ is $\varphi^*\colon\XX^*\to\RXX\colon x^*\mapsto
\sup_{x\in\XX}(\pair{x}{x^*}-\varphi(x))$. $\Gamma_0(\XX)$ denotes
the set of lower semicontinuous convex functions from $\XX$ to
$\RX$ with nonempty domain. If $\XX$ is a Hilbert space, 
the proximity operator of 
$\varphi\in\Gamma_0(\XX)$ at $x\in\XX$ is the unique
minimizer, denoted by $\prox_\varphi x$, of
$\varphi+\|x-\cdot\|^2/2$.
Given $p\in [1,\pinf]$, the $\ell^p$ norm on
$\RR^d$ is denoted by $|\cdot|_p$.
The associated dual norm is 
$|\cdot|_{q}$, where $1/p+1/q=1$.
$\RR_+^m$ and $\RR_{++}^m$ are the positive and strictly
positive $m$-dimensional orthant, respectively, and 
$\RR_-^m =-\RR_+^m$. 

\section{A class of norms}
\label{sec:norms-definition}
We establish the mathematical foundation of our framework, starting 
with a scheme to construct convex functions on $\YY$.

\begin{proposition}
\label{p:milanojfY6Tr-21}
Let $\XX$, $\YY$, and $\KK$ be reflexive real Banach spaces,
let $K$ be a nonempty closed convex cone in $\KK$, 
let $L\colon\XX\to\YY$ be linear and bounded, and
let $F\colon\XX\to\KK$ be $K$-convex in the sense that
\begin{equation}
\label{e:gino}
(\forall\alpha\in\zeroun)(\forall x\in\XX)(\forall y\in\XX)\quad
F\big(\alpha x+(1-\alpha) y\big)-\alpha F(x)-(1-\alpha)F(y)\in K.
\end{equation}
Let $h\colon\KK\to\RX$ be a convex function such that
$\dom h\cap\ran F\neq\emp$ and
\begin{equation}
\label{e:milanojfY6Tr-20}
(\forall z_1\in\ran F)(\forall z_2\in\ran F)\quad
\quad z_1-z_2\in K\quad\Rightarrow\quad h(z_1)\leq h(z_2).
\end{equation}
Define 
\begin{equation}
\label{e:general-primal}
\varphi\colon\YY\to\RXX\colon
y\mapsto\inf_{\substack{x\in\XX\\ Lx=y}}h\big(F(x)\big).
\end{equation}
Then the following hold:
\begin{enumerate}
\item
\label{p:milanojfY6Tr-21i}
$\varphi$ is convex.
\item
\label{p:milanojfY6Tr-21ii}
Suppose that $F$ is continuous, that the cone generated by 
$\ran L^*-\dom(h\circ F)^*$ is a closed vector subspace of $\XX^*$,
and that $h$ is lower semicontinuous.
Then $\varphi\in\Gamma_0(\YY)$.
\end{enumerate}
\end{proposition}
\begin{proof}
Set $f=h\circ F$.

\ref{p:milanojfY6Tr-21i}:
It is enough to show that $f$ is convex, as this will imply that 
$\varphi$ is likewise \cite[Proposition~12.36(ii)]{Livre1}. Let 
$\alpha\in\zeroun$, and let $x$ and $y$ be points in $\XX$. 
Combining \eqref{e:gino} and \eqref{e:milanojfY6Tr-20} yields
\begin{equation}
h\big(F(\alpha x+(1-\alpha) y)\big)
\leq h\big(\alpha F(x)+(1-\alpha)F(y)\big).
\end{equation}
Therefore, by convexity of $h$, we obtain
\begin{equation}
(h\circ F)\big(\alpha x+(1-\alpha)y\big)\leq\alpha(h\circ F)(x)
+(1-\alpha)(h\circ F)(y),
\end{equation}
which establishes the convexity of $f$.

\ref{p:milanojfY6Tr-21ii}:
We first derive from \cite[Lemma~1.28]{Livre1} that $f$ is lower 
semicontinuous. Thus, $f\in\Gamma_0(\XX)$ and, since the cone 
generated by $\ran L^*-\dom f^*$ is a closed vector subspace of 
$\XX^*$, it follows from \cite[Theorem~2.8.3(vii)]{Zali02} and the 
same arguments used in the Hilbertian case in 
\cite[Corollary~25.44(i)]{Livre1} that $\varphi\in\Gamma_0(\YY)$.
\end{proof}

We are now ready to define a class of norms induced by optimal
interpolation, which will be the main focus of the present paper.

\begin{assumption}
\label{a:milanjfY6Tr-22}
$\YY$ is a real Banach space and $m$ is a strictly positive 
integer. For every $j\in\{1,\ldots,m\}$, $\XX_j$ is a reflexive
real Banach space with norm $\|\cdot\|_j$. A generic element in 
$\XX=\XX_1\times\cdots\times\XX_m$ is denoted by 
$x=(x_1,\ldots,x_m)$. Furthermore:
\begin{enumerate}
\item
\label{a:milanjfY6Tr-22i}
$F\colon\XX\to\RR^m\colon x\mapsto(\|x_1\|_1,\ldots,\|x_m\|_m)$.
\item
\label{a:milanjfY6Tr-22ii}
$|||\cdot|||$ is a norm on $\RR^m$ which is monotone in the sense
that
\begin{equation}
\label{e:milanjfY6Tr-22b}
\big(\forall a\in\RR_+^m\big)\big(\forall b\in\RR_+^m\big)\quad
a-b\in\RR_-^m\quad\Rightarrow\quad |||a|||\leq |||b|||.
\end{equation}
\item
\label{a:milanjfY6Tr-22iii}
For every $j\in\{1,\ldots,m\}$, $L_j\in\BL(\XX_j,\YY)$,
$L\colon\XX\to\YY\colon x\mapsto L_1x_1+\cdots+L_mx_m$, and 
$\ran L=\YY$.
\end{enumerate}
Set
\begin{equation}
\label{e:0}
(\forall y\in\YY)\quad
\|y\|=\inf_{\substack{x\in\XX\\ Lx=y}}|||F(x)|||=
\inf_{\substack{x_1\in\XX_1,\ldots,x_m\in\XX_m\\ 
L_1x_1+\cdots+L_mx_m=y}}
~\big|\big|\big|\big(\|x_1\|_1,\ldots,\|x_m\|_m\big)\big|\big|\big|. 
\end{equation}
\end{assumption}

\begin{proposition}
\label{p:comp}
Consider the setting of Assumption~\ref{a:milanjfY6Tr-22} and set
$\nn=|||\cdot|||\circ F$. Then the following hold:
\begin{enumerate}
\item
\label{p:compi}
$\nn$ is a norm on $\XX$.
\item
\label{p:compii}
The dual norm of $\nn$ at $x^*\in\XX^*$ is 
$\nn_*(x^*)=|||(\|x^*_1\|_{1*},\dots,\|x^*_m\|_{m*})|||_*$.
\end{enumerate}
\end{proposition}
\begin{proof}
Let $x\in\XX$ and $x^*\in\XX^*$. 

\ref{p:compi}:
We first deduce from 
Assumption~\ref{a:milanjfY6Tr-22}\ref{a:milanjfY6Tr-22i} that 
\begin{align}
(\forall\alpha\in\RR)\quad
\nn(\alpha x)
&=|||(\|\alpha x_1\|_1,\ldots,\|\alpha x_m\|_m)|||\nonumber\\
&=|||(|\alpha|\,\|x_1\|_1,\ldots,|\alpha|\,\|x_m\|_m)|||\nonumber\\
&=|\alpha|\,|||(\|x_1\|_1,\ldots,\|x_m\|_m)|||\nonumber\\
&=|\alpha|\,|||F(x)|||\nonumber\\
&=|\alpha|\,\nn(x)
\end{align}
and that
\begin{eqnarray}
\nn(x)=0
&\Leftrightarrow&F(x)=0\nonumber\\
&\Leftrightarrow&(\forall j\in\{1,\ldots,m\})\;\;
\|x_j\|_j=0\nonumber\\
&\Leftrightarrow&(\forall j\in\{1,\ldots,m\})\;\; x_j=0\nonumber\\
&\Leftrightarrow& x=0.
\end{eqnarray}
To check the triangle inequality, let $z\in\XX$. By 
Assumption~\ref{a:milanjfY6Tr-22}\ref{a:milanjfY6Tr-22i}, 
$F(x+z)-F(x)-F(z)\in\RR_-^m$. Hence, we 
derive from \eqref{e:milanjfY6Tr-22b} that 
\begin{equation}
\nn(x+z)=|||F(x+z)|||\leq|||F(x)+F(z)|||\leq|||F(x)|||+|||F(z)|||
=\nn(x)+\nn(z).
\end{equation}

\ref{p:compii}:
Suppose that $\nn(x)=1$, set 
$b=(\|x^*_j\|_{j*})_{1\leq j\leq m}$, and observe that
\begin{equation}
\label{e:94fh}
\pair{x}{x^*}=\sum_{j=1}^m\pair{x_j}{x^*_j}\leq\sum_{j=1}^m
\|x_j\|_j\,\|x^*_j\|_{j*}={F(x)}\trans{b}\leq |||b|||_*.
\end{equation}
Taking the supremum over all such vectors $x$, we obtain 
$\nn_*(x^*)\leq|||b|||_*$.
However, by \eqref{e:dualnorm}, since $b\in\RP^m$, there exists 
$a=(\alpha_j)_{1\leq j\leq m}\in\RP^m$ such 
that $|||a|||=1$ and ${a}\trans{b}=|||b|||_*$.
Likewise, for every $j\in\{1,\ldots,m\}$, there exists
$z_j\in\XX_j$ such that $\|z_j\|_j=1$ and
$\|x^*_j\|_{j*}=\pair{z_j}{x^*_j}$. Now set
$\overline{x}=(\overline{x}_j)_{1\leq j\leq m}$, where
$(\forall j\in\{1,\ldots,m\})$ 
$\overline{x}_j=\alpha_jz_j$. Then 
\begin{equation}
\label{e:west4th}
\nn(\overline{x})
=\big|\big|\big|\big(\|\alpha_1z_1\|_1,\ldots,\|\alpha_mz_m\|_m\big)
\big|\big|\big|
=\big|\big|\big|\big(\alpha_1,\ldots,\alpha_m\big)
\big|\big|\big|
=|||a|||
=1
\end{equation}
and therefore 
\begin{align}
\nn_*(x^*)
=\sup_{\substack{w\in\XX\\ \nn(w)=1}}\pair{w}{x^*}
\geq\pair{\overline{x}}{x^*}
=\sum_{j=1}^m\pair{\overline{x}_j}{x^*_j}
=\sum_{j=1}^m\alpha_j\pair{z_j}{x^*_j}
=\sum_{j=1}^m\alpha_j\|x^*_j\|_{j*}
=a\trans b
=|||b|||_*.
\end{align}
We conclude that $\nn_*(x^*)=|||b|||_*$.
\end{proof}

\begin{remark}
\label{r:234c5v}
Let $y\in\YY$ and set $C=\menge{x\in\XX}{Lx=y}$. 
Since $\ran L=\YY$, we have $C\neq\emp$. Now let 
$d_C$ be the distance function to the affine subspace
$C$ associated with the norm 
$\nn=|||\cdot|||\circ F$ (see Proposition~\ref{p:comp}), that is,
\begin{equation}
\label{e:d_C}
(\forall z\in\XX)\quad d_C(z)=\inf_{x\in C}\nn(z-x).
\end{equation}
It follows from \eqref{e:0} that
\begin{equation}
\label{e:q3o94c513}
d_C(0)=\inf_{x\in C}\nn(x-0)=\inf_{x\in C}\nn(x)=\|y\|. 
\end{equation}
Thus, the function $\|\cdot\|$ in \eqref{e:0} is
defined via a minimal norm interpolation process, that is, the
optimization problem underlying \eqref{e:0} is that of
minimizing the norm $\nn$ over the affine subspace $C$. 
Optimal interpolation and, in particular, the problem of finding a
minimal norm interpolant to a finite set of points 
has a long history in approximation theory; see, e.g.,
\cite{Chen00} and the references therein. 
\end{remark}

In the next result we show that the construction described in
Assumption~\ref{a:milanjfY6Tr-22} does provide a norm, and we
compute its dual norm.

\begin{theorem}
\label{p:milanojfY6Tr-22}
Consider the setting of Assumption~\ref{a:milanjfY6Tr-22}.
Then the following hold:
\begin{enumerate}
\item
\label{p:milanojfY6Tr-22i}
$\|\cdot\|$ is a norm on $\YY$.
\item
\label{p:milanojfY6Tr-22ii}
Suppose that $\YY$ is finite-dimensional. Then the dual norm of 
$\|\cdot\|$ at $y^*\in\YY^*$ is 
\begin{equation}
\label{e:dual}
\|y^*\|_*=\big|\big|\big|\big(\|L_1^*y^*\|_{1*},\ldots,
\|L_m^*y^*\|_{m*}\big)\big|\big|\big|_*. 
\end{equation}
\end{enumerate}
\end{theorem}
\begin{proof}
Set $\nn=|||\cdot|||\circ F$ and recall from 
Proposition~\ref{p:comp} that $\nn$ is a norm.

\ref{p:milanojfY6Tr-22i}: 
We first note that, since $\ran L=\YY$, $\dom\|\cdot\|=\YY$.
Next, we derive from \eqref{e:0} that, for every
$y\in\YY$ and every $\alpha\in\RR\smallsetminus\{0\}$,
\begin{equation}
\label{e:milanojfY6Tr-22z}
\|\alpha y\|
=\inf_{\substack{x\in\XX\\Lx=\alpha y}} \nn(x)
=|\alpha|\inf_{\substack{x\in\XX\\ L(x/\alpha)=y}} \nn(x/\alpha)
=|\alpha|\,\|y\|.
\end{equation}
On the other hand, it is clear that $F$ satisfies \eqref{e:gino}
with $K=\RR_-^m$,
that $|||\cdot|||$ satisfies \eqref{e:milanojfY6Tr-20}, and 
that \eqref{e:0} assumes the same form as 
\eqref{e:general-primal}. Hence, by 
Proposition~\ref{p:milanojfY6Tr-21}\ref{p:milanojfY6Tr-21i}, the 
function $\|\cdot\|$ is convex. 
In view of \eqref{e:milanojfY6Tr-22z}, we 
therefore have, for every $(y,z)\in\YY\times\YY$,
$\|y+z\|\leq\|y\|+\|z\|$. 
Now let $y\in\YY$ be such that $\|y\|=0$ and 
set $C=\menge{x\in\XX}{Lx=y}$. Then it follows from 
\eqref{e:q3o94c513} that $d_C(0)=0$ and, since $C$ is closed, we get
$0\in C$. Therefore, $y=L0=0$. Altogether, $\|\cdot\|$ is a norm.

\ref{p:milanojfY6Tr-22ii}: 
Let $y^*\in\YY^*$. Then
\begin{align}
\|y^*\|_*
&=\max\menge{\pair{y}{y^*}}{y\in\YY,\,\|y\|= 1}\nonumber\\
&=\max\Menge{\pair{y}{y^*}}{y\in\YY,\,
\min_{\substack{x\in\XX,\,Lx=y}}\nn(x)=1}\nonumber\\
&=\max\menge{\pair{Lx}{y^*}}{x\in\HH,\,\nn(x)= 1}\nonumber\\
&=\max\menge{\pair{x}{L^* y^*}}{x\in\HH,\,\nn(x)= 1}
\nonumber\\
&=\nn_*(L^*y^*).
\end{align}
We conclude by applying Proposition~\ref{p:comp}.
\end{proof}

We illustrate the construction \eqref{e:0} via two examples.

\begin{example}
\label{ex:peetre}
In Theorem~\ref{p:milanojfY6Tr-22} suppose that $m=2$, that
$\XX_1$ and $\XX_2$ are continuously embedded in the same
topological vector space $\YY$, that $L_1$ and $L_2$ are the
canonical injections, and that $|||\cdot|||=|\cdot|_1$.
Then \eqref{e:0} becomes
\begin{equation}
(\forall y\in\YY)\quad
\|y\|=\min_{\substack{x_1\in\XX_1,x_2\in\XX_2\\ x_1+x_2=y}}
~\big(\|x_1\|_1+\|x_2\|_2\big).
\end{equation}
In other words, $\|\cdot\|$ represents the infimal convolution of
the norms $\|\cdot\|_1$ and $\|\cdot\|_2$. This type of construct
is central is the theory of interpolation spaces
\cite{Peet70,Trie78}. If we replace the $\ell^1$ norm by the 
$\ell^p$ norm for some $p\in\left]1,\pinf\right[$ above, 
we obtain, 
\begin{equation}
(\forall y\in\YY)\quad
\|y\|=\min_{\substack{x_1\in\YY_1,\,x_2\in\YY_2\\ x_1+x_2=y}}
~\big(\|x_1\|^p_1+\|x_2\|^p_2\big)^{1/p}.
\end{equation}
This formulation also arises in the area of interpolation spaces 
\cite{Goul68,Peet70}.
\end{example}

\begin{example}
\label{ex:1}
Let $\HH$ be a real Hilbert space with norm 
$\|\cdot\|_{\HH}$, which is identified with its dual. 
In Theorem~\ref{p:milanojfY6Tr-22} suppose that 
$\XX_1=\cdots=\XX_m=\HH$, let $p$ and $q$ be numbers in 
$]1,\pinf[$ such that $1/p+1/q=1$, and let 
$|||\cdot|||=|\cdot|_p$. Then \eqref{e:0} becomes
\begin{equation}
(\forall y\in\YY)\quad
\|y\|=\left(\min_{\substack{x_1\in\HH,\ldots,x_m\in\HH\\ 
\sum_{j=1}^{m}L_jx_j=y}}\:\sum_{j=1}^m\|x_j\|_{\HH}^p\right)^{1/p}.
\end{equation}
Furthermore, if $\HH$ is finite dimensional, the dual norm at
$y^*\in\YY$ is given by \eqref{e:dual} as
\begin{equation}
\|y^*\|_*=\left(\sum_{j=1}^{m}\big\|L_j^*y^*
\big\|^{q}_{\HH}\right)^{1/q}.
\end{equation}
This construction is discussed in \cite[Theorem~7]{Maur12}.
\end{example}

\begin{remark}
\label{r:milanojfY6Tr-21}
Any norm $\|\cdot\|$ on $\YY$ can trivially be written in the form of
\eqref{e:0} by letting $m=1$, $\XX_1=\YY$, $\|\cdot\|_1=\|\cdot\|$,
$L_1=\Id$, and $|||\cdot|||=|\cdot|$. However,
we are interested in exploiting the structure of the construction
\eqref{e:0} in cases in which the norms $|||\cdot|||$
and $(\|\cdot\|_j)_{1\leq j\leq m}$ are chosen from a ``simple''
class and give rise, via the optimal interpolation problem
\eqref{e:0}, to a ``complex'' norm $\|\cdot\|$. In
particular, when using proximal splitting methods, the
computation of $\prox_{\|\cdot\|}$ will typically not be easy
whereas that of the operators
$(\prox_{\|\cdot\|_j})_{1\leq j\leq m}$ will be.
This will be exploited in Section~\ref{sec:opt} to devise
an efficient block-coordinate splitting algorithm
in the case when $|||\cdot|||=|\cdot|_1$.
\end{remark}

\section{Examples}
\label{sec:examples}

In this section, we observe that the construct presented in 
Assumption~\ref{a:milanjfY6Tr-22} contains in a single framework a 
number of existing regularizers. For simplicity, we
focus on the norms captured by Example~\ref{ex:1}. Our main aim
here is not to derive new regularizers but, rather, to show that
our analysis captures existing ones and to derive their dual norms.

\subsection{Latent group lasso}

\begin{notation}
\label{not:2}
The support of $y=(\eta_i)_{1\leq i\leq d}\in\RR^d$ is 
$\operatorname{supp}(y)=\menge{i\in\{1,\ldots,d\}}{\eta_i\neq 0}$.
For every $\emp\neq G\subset\{1,\dots,d\}$ and $y\in\RR^d$, we set
$r=\card G$ and let $y|_{G}$ denote the vector in $\RR^r$
obtained by retaining the components of $y$ indexed by $G$, i.e.,
$y|_{G}=(\eta_i)_{i\in G}$. Finally, $e_i$ is the $i$th standard
unit vector in $\RR^d$. 
\end{notation}

The example we consider is known as the latent group lasso 
(LGL), or group lasso with overlap, which goes back to 
\cite{Jacob2009-GL}. For every $j\in\{1,\dots,m\}$, fix
$(p_j,q_j) \in[1,\pinf] \times [1,\pinf]$ such that $1/p_j +
1/q_j=1$. Let $(G_j)_{1\leq j\leq m}$ be a covering of 
$\{1,\dots,d\}$ and define the vector space
\begin{equation}
Z=\menge{(z_j)_{1\leq j\leq m}}{(\forall j\in\{1,\ldots,m\})\;
z_j\in\RR^d\;\text{and}\;\operatorname{supp}(z_j)\subset G_j}. 
\end{equation}
The latent group lasso penalty is defined, for $y\in\RR^d$, as 
\begin{equation}
\label{e:latent-group-lasso}
\|y\|_{\textrm{LGL}}=\min\MEnge{\sum_{j=1}^m|z_j|_{p_j}}
{(z_j)_{1\leq j\leq m}\in Z,~{\sum\limits_{j=1}^m z_j=y}}.
\end{equation}
The optimal interpolation problem \eqref{e:latent-group-lasso}
seeks a decomposition of vector $y$ in terms of vectors
$(z_j)_{1\leq j\leq m}$ the support sets of which are restricted 
to the corresponding group of variables in $G_j$. 
If the groups overlap
then the decomposition is not 
necessarily unique, and the variational formulation involves those
$z_j$ for which $\sum_{j=1}^m|z_j|_{p_j}$ is minimal. On the
other hand, if the groups 
are pairwise disjoint, that is $(G_j)_{1\leq j\leq m}$ forms 
a partition of $\{1,\dots,d\}$, the latent group lasso norm 
coincides with the ``standard'' group lasso norm \cite{Yuan2006}, 
which is defined as
\begin{equation}
\label{e:group-lasso}
\big(\forall y\in\RR^d\big)\quad
\|y\|_{\textrm{GL}}=\sum_{j=1}^m\big|y_{|G_j}\big|_{p_j}.
\end{equation}
The norms \eqref{e:latent-group-lasso} 
and \eqref{e:group-lasso} are presented in \cite{Jacob2009-GL} and
\cite{Yuan2006} respectively in the case that $p_1=\cdots=p_m=p$ 
and $q_1=\cdots=q_m=q$. 
In general, \eqref{e:latent-group-lasso} has no closed form
expression due to the overlapping of the groups. However, in
special cases which exhibit additional structure, it can be
computed in a finite number of steps. An important example is
provided by the $(k,p)$-support norm \cite{McDonald2016}, in
which the groups consist of all subsets of $\{1,\dots,d\}$ of
cardinality no greater than $k$, for some $k \in \{1,\dots,d\}$.
The case $p=2$ has been studied in \cite{Argyriou2012} and
\cite{McDonald2014}. 

\begin{example}
\label{ex:iit}
In the above setting, for every $j\in\{1,\dots,m\}$, set
$r_j=\card G_j$. The latent group
lasso penalty \eqref{e:latent-group-lasso} is a norm of the form
\eqref{e:0} with $\YY=\RR^d$, $|||\cdot|||=|\cdot|_1$, and,
for every $j\in\{1,\ldots,m\}$, $\XX_j=\RR^{r_j}$, 
$\|\cdot\|_j=|\cdot|_{p_j}$ and 
$L_j=[e_i~|~i \in G_j]$ is a $d \times r_j$ matrix.
The change of variables $z_j=L_j x_j$ then yields
\eqref{e:latent-group-lasso}.
Furthermore, the dual norm is given by
\begin{equation}
\label{e:kkkk}
\big(\forall y^*\in\RR^d\big)\quad
\|y^*\|_{{\text{LGL}}*}=\max_{1\leq j\leq m}\|y^*_{|G_j}\|_{q_j}.
\end{equation}
This follows from \eqref{e:dual} by noting
that $|||\cdot|||_*=|\cdot|_\infty$ and, 
for all $j\in\{1\ldots,m\}$, $\|\cdot\|_{j*}=|\cdot|_{q_j}$ and
$L_j^*y^*={y^*}_{|G_j}$.
\end{example}

\subsection{Overlapping group lasso}
An alternative generalization of the group lasso norm
\eqref{e:group-lasso} is the overlapping
group lasso \cite{Jenatton2011,Zhao09}, which we denote by
$\|\cdot\|_{\textrm{OGL}}$. It has the same expression as 
\eqref{e:group-lasso}, except that we drop the
restriction that the groups $(G_j)_{1\leq j\leq m}$ form a
partition of $\{1,\dots,d\}$. 
Our next result establishes that the overlapping group lasso
penalty is captured by a dual norm of type \eqref{e:dual}.
We continue to use Notation~\ref{not:2}.

\begin{example}
\label{prop:ogl}
For every $j\in\{1,\dots,m\}$, let $(p_j,q_j) \in[1,\pinf]
\times [1,\pinf]$ be such that $1/p_j+1/q_j=1$. Let
$(G_j)_{1\leq j\leq m}$ be a covering of $\{1,\ldots,d\}$
and let $|||\cdot|||=|\cdot|_\infty$. For every
$j\in\{1,\ldots,m\}$, set $r_j=\card G_j$, 
$\|\cdot\|_{j}=|\cdot|_{q_j}$, and $L_j=[e_i~|~i \in G_j]$.
Then the norm \eqref{e:0} evaluated at $y\in\RR^d$ is
\begin{equation}
\|y\|=\inf\MEnge{\max_{1\leq j\leq m}
|x_{j}|_{q_j}}{ \sum_{j=1}^m L_j x_{j}=y}.
\end{equation}
In addition, since 
$|||\cdot|||_*=|\cdot|_1$ and, for every $j\in\{1,\dots,m\}$,
$\|\cdot\|_{j*}=|\cdot|_{p_j}$ and
$|L_j\trans y^*|_{p_j} =|y^*_{|G_j}|_{p_j}$, 
\eqref{e:dual} yields
$\|\cdot\|_*=\|\cdot\|_{\textrm{OGL}}$.
\end{example}

\begin{remark}
The case $p=\pinf$ corresponds to the iCAP penalty of
\cite{Zhao09}. We can also consider other choices for the
matrices $(L_j)_{1\leq j\leq m}$.
For example, an appropriate choice gives 
various total variation penalties \cite{MSX}. 
A further example is obtained 
by choosing $m=1$ and $L_1$ to be the incidence 
matrix of a graph, a setting which is has been considered in the
context of semi-supervised learning~\cite{Herb09}. In
particular, for $p=1$, this corresponds to the fused lasso penalty
\cite{fused}.
\end{remark}

\begin{remark}
We have seen that the norm \eqref{e:0} captures the
latent group lasso when $|||\cdot|||=|\cdot|_1$, while the dual
norm \eqref{e:dual} captures the overlapping group lasso when
$|||\cdot|||=|\cdot|_\infty$. A natural extension of either
setting is to choose $|||\cdot|||$ to be the $k$ Ky-Fan norm for
some $k \in \{1,\dots,m\}$. In other words we choose, for every 
$x\in\RR^m$, $|||x|||= \sum_{j=1}^k |x|^\downarrow_j$, where 
$|x|^\downarrow \in \RR^m$ is the vector obtained by
reordering the components of $x$ so that they are decreasing in
absolute value. Note that, if $k=1$, then
$|||\cdot|||=|\cdot|_\infty$ and, if $k=m$, then
$|||\cdot|||=|\cdot|_1$.
\end{remark}

\subsection{Polyhedral norms}
A norm on $\YY=\RR^d$ is polyhedral (or is a block-norm) if its 
closed unit ball $B$ is a polyhedron, i.e., a finite intersection of 
closed affine half-spaces. Examples of polyhedral norms used in
data processing can be found in \cite{McDonald2016,Zeng14,Zhao09}. 
In this case, $B$ is bounded, symmetric
with respect to the origin, and it has a finite, even number of 
extreme points. Let us recall a couple of useful facts.

\begin{fact}{\rm\cite[Theorem~1]{Ward85}}
\label{f:1}
Let $\|\cdot\|$ be a polyhedral norm on $\YY=\RR^d$ and let 
$(b_j)_{1\leq j\leq m}$ be the extreme points of its closed unit
ball. Then
\begin{equation}
\label{e:31}
(\forall y\in\YY)\quad\|y\|=\min_{\substack{(\xi_j)_{1\leq j\leq m}
\in\RR^m\\
\sum_{j=1}^m \xi_j b_j=y}}\;\sum_{j=1}^m |\xi_j|.
\end{equation}
\end{fact}

\begin{fact}{\rm\cite[Theorem~2]{Ward85}}
\label{f:2}
Let $\|\cdot\|$ be a polyhedral norm on $\YY=\RR^d$, let $B$ be its
closed unit ball, and let 
$B^\odot=\menge{y^*\in\RR^d}{(\forall y\in B)\;\;y\trans{y^*}\leq 1}$
be the polar set of $B$. Then $B^\odot$ is a bounded polyhedron
and, if $(b_j^*)_{1\leq j\leq r}$ denote its extreme points, 
\begin{equation}
\label{e:32}
\big(\forall y\in\RR^d\big)\quad\|y\|=
\max_{1\leq j\leq r}y\trans b_j^*.
\end{equation}
\end{fact}

It follows from Fact~\ref{f:1} that polyhedral norms are 
special cases of \eqref{e:0}. Indeed, \eqref{e:31} is derived from 
\eqref{e:0} by choosing $|||\cdot|||=|\cdot|_1$ and $(\forall
j\in\{1,\ldots,m\})$ $\XX_j=\RR$, $\|\cdot\|_j=|\cdot|$, and
$L_j\colon\xi\mapsto\xi b_j$. In addition, since a linear 
function on a nonempty compact convex set attains its maximum at 
an extreme point of the set \cite[Corollary~32.3.2]{Rock70}, the
dual norm is given by
\begin{equation}
\label{e:33}
\big(\forall y^*\in\RR^d\big)\quad\|y^*\|_*
=\max_{y\in B}y\trans y^*
=\max_{1\leq j\leq m}b_j\trans y^*.
\end{equation}
Note that \eqref{e:33} is the dual counterpart of \eqref{e:32}.

\subsection{$\Theta$-norms}

Assume that $\YY=\RR^d$. Families of norms parameterized by a
nonempty, convex, and bounded set $\Theta\subset\RR^d_{++}$ were 
considered in \cite{Bach2012,McDonald2014,Micc13}. As
shown in \cite[Proposition~2]{Macd16}, the expressions 
\begin{equation}
\label{e:ciccio1}
\big(\forall y\in\RR^d\big)\quad\|y\|_{\Theta}
=\inf_{\theta\in\Theta}\sqrt{\scal{y}{\diag(\theta)^{-1}y}}
\end{equation}
and
\begin{equation}
\label{e:ciccio2}
\big(\forall y^*\in\RR^d\big)\quad\|y^*\|_{\Theta*}
=\sup_{\theta\in\Theta}\sqrt{\scal{y^*}{\diag(\theta)y^*}}
\end{equation}
define dual norms. 
Examples of norms which are included in this 
family are the $\ell^p$-norms, and the $k$-support norm mentioned 
above \cite{Argyriou2012}. 
Next, we relate $\Theta$-norms to our framework.

\begin{proposition}
\label{prop:theta}
Let $(\theta_j)_{1\leq j\leq m}$ be vectors in $\RR^d_{+}$ and
let $\Theta$ be a subset of~$\RR^d_{++}$ such that
$\overline{\Theta}=\conv\{\theta_{1},\dots,\theta_{m}\}$. Then
the norms defined in \eqref{e:ciccio1} and
\eqref{e:ciccio2} can be written in the form \eqref{e:0}
and \eqref{e:dual} respectively, with $|||\cdot|||=|\cdot|_1$
and $(\forall j\in\{1,\ldots,m\})$
$\|\cdot\|_j=|\cdot|_2$ and $L_j=\diag(\sqrt{\theta_{j}})$, where
the $\sqrt{\cdot}$ operator is understood componentwise. 
\end{proposition}
\begin{proof}
We have $|||\cdot|||_*=|\cdot|_{\infty}$ and
$(\forall j\in\{1,\ldots,m\})$ $\|\cdot\|_{j*}=|\cdot|_2$. 
Now set $(\forall j\in \{1,\ldots,m\})$
$L_j=\diag(\sqrt{\theta_{j}})$ and $\XX_j=\RR^d$. 
Then we derive from \eqref{e:dual} that
\begin{align}
\big(\forall y^*\in\RR^d\big)\quad 
\|y^*\|^2_* 
&=\max_{1\leq j\leq m} | L_j\trans y^*|^2_2
\nonumber\\ 
&=\max_{1\leq j\leq m}{\scal{y^*}{\diag(\theta_{j})y^*}}
\nonumber\\
&=\max_{\theta\in\overline{\Theta}}
{\scal{y^*}{\diag(\theta)y^*}}\label{e:att} \\
&=\sup_{\theta\in\Theta}{\scal{y^*}{\diag(\theta)y^*}},
\end{align}
where equality \eqref{e:att} results from the fact that, in 
$\RR^d$, a linear 
function on a nonempty compact convex set attains its maximum at 
an extreme point of the set \cite[Corollary~32.3.2]{Rock70}. This
establishes \eqref{e:ciccio2}. As noted above, \eqref{e:ciccio2} is 
the dual of \eqref{e:ciccio1}. It follows that  \eqref{e:ciccio1} is 
of the form \eqref{e:0}.
\end{proof}

\subsection{Tensor norms}
A number of regularizers have been proposed to learn low
rank tensors; see
\cite{Gandy11,Romera-Paredes2013,Signoretto14,Tomi11,Wima14}
and the references therein. In this section, we discuss two 
prominent examples that fit our framework. We
first recall some notions from multilinear algebra \cite{Kolda}.
Let $\W\in\RR^{d_1\times\cdots\times d_m}$ be an $m$-mode 
real tensor, that is,
\begin{equation}
\W=\left[Y_{i_1,\dots,i_m}\right]_{1\leq i_1\leq d_1,\ldots,
1\leq i_m\leq d_m}.
\end{equation}
Now let $j\in\{1,\dots,m\}$. A mode-$j$ fiber is a vector
composed of elements of $\W$ obtained
by fixing all indices except those corresponding to the $j$th.
Set $r_j=\prod_{k \neq j} d_k$. The mode-$j$ matricization
$M_j(\W)$ of a tensor $\W$ is the $d_j \times r_j$ matrix
obtained by arranging the mode-$j$ fibers of $\W$ such that each 
of them forms a column of $M_j(\W)$. By way of example, a 3-mode
tensor $\W\in\RR^{3\times 4\times 2}$ admits the matricizations: 
$M_1(\W)\in\RR^{3\times 8}$, $M_2(\W)\in\RR^{4\times 6}$, and 
$M_3(\W)\in\RR^{2\times 12}$. Note that $M_j\colon\RR^{d_1
\times\cdots\times d_m} \to \RR^{d_j \times r_j}$ is a
linear operator. Its adjoint
$M_j^*\colon\RR^{d_j \times r_j}
\to \RR^{d_1 \times\cdots\times d_m}$ is the reverse
matricization along mode $j$. 

Recall that the nuclear norm (or trace norm) of a matrix,
$\|\cdot\|_{\textrm{nuc}}$, is the sum of its singular values. Its
dual norm is the spectral norm $\|\cdot\|_{\textrm{sp}}$ which
provides the largest singular value. The overlapped nuclear norm
\cite{Romera-Paredes2013,Tomi11} is defined as the sum of the
nuclear norms of the mode-$j$ matricizations, namely
\begin{align}
\label{e:overlapped-trace-norm}
\big(\forall \W\in\RR^{d_1 \times \cdots \times d_m}\big)\quad
\Vert \W
\Vert_{\textrm{ONN}}=\sum_{j=1}^m \Vert M_j(\W)
\Vert_{\textrm{nuc}}.
\end{align}

\begin{example}
\label{ex:onn}
Let $\YY=\RR^{d_1 \times \cdots \times d_m}$ and 
$|||\cdot|||=|\cdot|_\infty$. In addition, for every 
$j\in\{1,\dots,m\}$, set $\XX_j=\RR^{d_j \times r_j}$,
$\|\cdot\|_{j}=\|\cdot\|_{\textrm{sp}}$, and $L_j=M_j^*$.
Then \eqref{e:0} becomes
\begin{equation}
\big(\forall \W\in\RR^{d_1\times\cdots\times d_m}\big)\quad
\|\W\|=\inf\MEnge{\max_{1\leq j\leq m}
\|X_{j}\|_{\textrm{sp}}}{\sum_{j=1}^m M_j^* X_{j}=\W}.
\end{equation}
In addition, since 
$|||\cdot|||_*=|\cdot|_1$ and, for every $j\in\{1,\dots,m\}$,
$\|\cdot\|_{j*}=\|\cdot\|_{\textrm{nuc}}$ and
$L_j^* =M_j$, equation~\eqref{e:dual} yields 
$\|\cdot\|_*=\|\cdot\|_{\textrm{ONN}}$.
\end{example}

Now let $\{\alpha_j\}_{1\leq j\leq m}\subset\RPP$. The
scaled latent nuclear norm is defined by (see \cite{Tomi13} and
\cite{Wima14} for special cases)
\begin{equation}
\big(\forall \W\in\RR^{d_1\times\cdots\times d_m}\big)\quad\Vert\W
\Vert_{\textrm{LNN}}=\inf\MEnge{\sum_{j=1}^m 
\frac{1}{\alpha_j}\|X_{j}\|_{\textrm{nuc}}}
{\sum_{j=1}^m M_j^* X_j=\W}.
\label{e:latent-trace-norm}
\end{equation}
Our next example captures this norm.

\begin{example}
\label{ex:slnn}
The latent nuclear norm \eqref{e:latent-trace-norm} is of the form
\eqref{e:0} with $\YY=\RR^{d_1 \times \cdots \times d_m}$,
$|||\cdot|||=|\cdot|_1$ and, for every $j\in\{1,\dots,m\}$,
$\XX_j= \RR^{d_j \times r_j}$, $\|\cdot\|_{j}=
\|\cdot\|_{\textrm{nuc}}/\alpha_j$, and $L_j=M_j^*$. 
Furthermore, \eqref{e:dual} yields the dual norm 
\begin{equation}
\big(\forall \W^*\in\RR^{d_1 \times \cdots \times d_m}\big)
\quad\Vert\W^*\Vert_{\textrm{LNN}*}=\max_{1\leq j \leq m} 
\alpha_j \Vert M_j(\W^*)\Vert_{\textrm{sp}}.
\end{equation}
\end{example}

\section{Random block-coordinate algorithm}
\label{sec:opt}

\subsection{Overview}
The purpose of this section is to address some of the numerical 
aspects associated with the class of norms introduced in 
Assumption~\ref{a:milanjfY6Tr-22} in the case when
$|||\cdot|||=|\cdot|_1$, which reduces \eqref{e:0} to
\begin{equation}
\label{e:10}
(\forall y\in\YY)\quad
\|y\|=\inf_{\substack{x_1\in\XX_1,\ldots,x_m\in\XX_m\\ 
L_1x_1+\cdots+L_mx_m=y}}
~\|x_1\|_1+\cdots+\|x_m\|_m. 
\end{equation}
Since such norms are 
nonsmooth convex functions, they could in principle be handled via 
their proximity operators in the context of proximal splitting
algorithms \cite{Livre1,Siop13}. However, the
proximity operator of the composite norm $\|\cdot\|$ in 
\eqref{e:10} is usually intractable, which makes this direct 
approach unviable. We circumvent this problem by formulating 
the problem in such a way that it involves only the proximity 
operators of the norms $(\|\cdot\|_j)_{1\leq j\leq m}$, which will 
typically be available in closed form. 

The main features of the algorithmic approach we propose are the
following:
\begin{itemize}
\item
It can handle general nonsmooth formulations: the functions 
present in the model need not be differentiable.
\item
It adapts the recent approach proposed in \cite{Siop15,MaPr17} to 
devise a block-coordinate algorithm which allows us to select
arbitrarily the blocks of norms $(\|\cdot\|_j)_{1\leq j\leq m}$ to
be activated over the course of the iterations. This makes the
method amenable to the processing of very large data sets in a
flexible manner by adapting the computational load of each
iteration to the available computing resources.
\item
The computations are broken down to the evaluation of simple
proximity operators of the norms $(\|\cdot\|_j)_{1\leq j\leq m}$ 
and of those appearing in the loss function, while the linear
operators $(L_j)_{1\leq j\leq m}$ are applied separately.
\item
Knowledge of the norms of the operators $(L_j)_{1\leq j\leq m}$
is not required.
\item
Convergence of the iterates to a solution of the
minimization problem under consideration is guaranteed.
\end{itemize}

\subsection{Problem formulation}
We consider the standard linear problem in which a vector
$\overline{y}$ in a real Hilbert space $\YY$ is to be inferred
from  $n$ noisy linear observations
\begin{equation}
\label{e:genn23}
\begin{cases}
\beta_1&=\pair{\overline{y}}{a^*_1}+\zeta_1\\
&~\vdots\\
\beta_n&=\pair{\overline{y}}{a^*_n}+\zeta_n,
\end{cases}
\end{equation}
where $(a^*_i)_{1\leq i\leq n}\in(\YY^*)^n$ are known and 
$(\zeta_i)_{1\leq i\leq n}\in\RR^n$ model unknown perturbations.
This model captures various problems in 
supervised learning and in inverse problems. A common variational 
formulation associated with \eqref{e:genn23} is the 
regularized convex minimization problem 
\begin{equation}
\label{e:jfY6Tr-31o}
\minimize{y\in\YY}{\sum_{i=1}^n\ell_i(\pair{y}{a^*_i},\beta_i)+
\lambda\|y\|},
\end{equation}
where $\|\cdot\|$ is a norm fulfilling 
Assumption~\ref{a:milanjfY6Tr-22} with $|||\cdot|||=|\cdot|_1$ 
(see \eqref{e:10}), $\lambda\in\RPP$ is a regularization parameter,
and, for every $i\in\{1,\dots,n\}$ and every $\beta\in\RR$,
$\ell_i(\cdot,\beta)\in\Gamma_0(\RR)$. We also assume that each
$\XX_j$ is finite-dimensional and can be equipped with a norm
$\|\cdot\|_{\textrm{j}}$ that makes it a Euclidean space. We
designate by $\HH$ the Euclidean space obtained by renorming
$\XX=\XX_1\times\cdots\times\XX_m$ with the norm
$x=(x_j)_{1\leq j\leq m}\mapsto
\sqrt{\sum_{j=1}^m\|x_j\|_{j}^2}$. 
In this setting, \eqref{e:jfY6Tr-31o} becomes
\begin{equation}
\label{e:jfY6Tr-31a}
\minimize{\substack{y\in\YY\\ x_1\in\XX_1,\ldots,x_m\in\XX_m\\
\sum_{j=1}^mL_jx_j=y}}
{\Sum_{i=1}^n\ell_i(\pair{y}{a^*_i},\beta_i)
+\lambda\sum_{j=1}^m\|x_j\|_j}.
\end{equation}
One can therefore first obtain a solution 
$(x_j)_{1\leq j\leq m}$ to the problem 
\begin{equation}
\label{e:kalas}
\minimize{\substack{x_1\in\XX_1,\ldots,x_m\in\XX_m}}
{\Sum_{i=1}^n\ell_i\Bigg(\Sum_{j=1}^m\pair{L_jx_j}{a^*_i},\beta_i
\Bigg)+\lambda\sum_{j=1}^m\|x_j\|_j}
\end{equation}
and set $y=\sum_{j=1}^mL_jx_j$ to obtain a solution
to \eqref{e:jfY6Tr-31a}.
To make the structure of \eqref{e:kalas} more apparent, let us
introduce the functions
\begin{equation}
\label{e:9hfe6-01b}
\Phi\colon\XX\to\RX\colon x\mapsto\lambda\sum_{j=1}^m\|x_j\|_j
\end{equation}
and
\begin{equation}
\label{e:9hfe6-01a}
\Psi\colon\RR^m\to\RX\colon(\eta_1,\ldots,\eta_n)\mapsto
\sum_{i=1}^n\psi_i(\eta_i),
\end{equation}
where, for every $i\in\{1,\ldots,n\}$,
\begin{equation}
\psi_i\colon\RR\to\RX\colon\eta_i\mapsto\ell_i(\eta_i,\beta_i).
\end{equation}
Let us also define 
\begin{equation}
A\colon\YY\to\RR^n\colon y\mapsto(\pair{y}{a^*_i})_{1\leq i\leq n}
\quad\text{and}\quad
B\colon\HH\to\RR^n\colon x\mapsto\sum_{j=1}^mB_jx_j,
\end{equation}
where, for every $j\in\{1,\ldots,m\}$,
\begin{equation}
B_j=AL_j\in\BL(\XX_j,\RR^n).
\end{equation}
Then, recalling from Assumption~\ref{a:milanjfY6Tr-22} that 
$L\colon\HH\to\YY\colon x\mapsto\sum_{j=1}^mL_jx_j$, we have 
$B=AL\in\BL(\HH,\RR^n)$ and we can thus rewrite
\eqref{e:jfY6Tr-31a} as
\begin{equation}
\label{e:9hfe6-03a}
\minimize{x\in\HH}\Phi(x)+\Psi(Bx).
\end{equation}
Note that our hypotheses imply that $\Phi\in\Gamma_0(\HH)$ and 
$\Psi\in\Gamma_0(\RR^n)$.

\subsection{Douglas-Rachford splitting in a product space}
We work in the direct Hilbert sum $\HHH=\HH\oplus\RR^n$.
Let us introduce the functions 
\begin{equation}
\begin{cases}
\boldsymbol{F}\colon\HHH\to\RX\colon
(x,r)\mapsto\Phi(x)+\Psi(r)\\\
\boldsymbol{G}=\iota_{\boldsymbol{V}},\quad\text{where}\quad 
\boldsymbol{V}=\gra B=\menge{(x,r)\in\HHH}{Bx=r}. 
\end{cases}
\end{equation}
Using the variable $\boldsymbol{x}=(x,r)$, we reduce
\eqref{e:9hfe6-03a} to the problem 
\begin{equation}
\label{e:9hfe6-03d}
\minimize{\boldsymbol{x}\in\HHH} 
\boldsymbol{F}(\boldsymbol{x})+\boldsymbol{G}(\boldsymbol{x})
\end{equation}
involving the sum of two functions in $\Gamma_0(\HHH)$ 
and which can be solved with the Douglas-Rachford algorithm
\cite[Section~27.2]{Livre1}. Let $\boldsymbol{y}_0\in\HHH$,
let $\gamma\in\RPP$, and let
$(\mu_k)_{k\in\NN }$ be a sequence
in $\left]0,2\right[$ such that
$\sum_{k\in\NN}\mu_k(2-\mu_k)=\pinf$.
The Douglas-Rachford algorithm 
\begin{equation}
\label{e:mer-egee07-15}
\begin{array}{l}
\text{for}\;k=0,1,\ldots\\
\left\lfloor
\begin{array}{l}
\boldsymbol{x}_{k}=\prox_{\gamma\boldsymbol{G}}\boldsymbol{y}_k\\
\boldsymbol{z}_{k}=
\prox_{\gamma\boldsymbol{F}}(2\boldsymbol{x}_{k}-\boldsymbol{y}_k)
\\
\boldsymbol{y}_{k+1}=\boldsymbol{y}_k+
\mu_k(\boldsymbol{z}_{k}-\boldsymbol{x}_{k})
\end{array}
\right.\\[2mm]
\end{array}
\end{equation}
produces a sequence $(\boldsymbol{x}_k)_{k\in\NN}$ which converges 
to a solution to \eqref{e:9hfe6-03d} 
\cite[Corollary~27.4]{Livre1}. However,
by \cite[Proposition~24.11 and Example~29.19(i)]{Livre1},
\begin{equation}
\label{e:9hfe6-03e}
\prox_{\boldsymbol{F}}\colon (x,r)\mapsto
\big(\prox_{\Phi}x,\prox_{\Psi}r\big),
\end{equation}
and
\begin{equation}
\label{e:9hfe6-03f}
\prox_{\boldsymbol{G}}\colon (u,s)\mapsto (x,Bx),
\quad\text{where}\;\;
x=u-B^*(\Id+BB^*)^{-1}(Bu-s)
\end{equation}
is the projection operator onto $\boldsymbol{V}$. Hence, 
upon setting $R=B^*(\Id+BB^*)^{-1}$, we can rewrite
\eqref{e:mer-egee07-15} as
\begin{equation}
\label{e:9hfe6-03g}
\begin{array}{l}
\text{for}\;k=0,1,\ldots\\
\left\lfloor
\begin{array}{l}
q_k=Bu_k-s_k\\
x_{k}=u_k-Rq_k\\
r_{k}=Bx_{k}\\
v_{k}=\prox_{\gamma{\Phi}}(2x_{k}-u_k)\\
t_{k}=\prox_{\gamma{\Psi}}(2r_{k}-s_k)\\
u_{k+1}=u_{k}+\mu_k(v_{k}-x_{k})\\
s_{k+1}=s_{k}+\mu_k(t_{k}-r_{k}),
\end{array}
\right.\\[2mm]
\end{array}
\end{equation}
where we have set 
$\boldsymbol{x}_k=(x_k,r_k)$,
$\boldsymbol{y}_k=(u_k,s_k)$, and
$\boldsymbol{z}_k=(v_k,t_k)$.
It follows from the above result that $(x_k)_{k\in\NN}$ converges
to a solution to \eqref{e:9hfe6-03a}. Let us now express 
\eqref{e:9hfe6-03g} in terms of the original variables of 
problem \eqref{e:kalas}.
To this end set, for every $j\in\{1,\ldots,m\}$, 
\begin{equation}
\label{e:Rj}
R_j=B_j^*(\Id+BB^*)^{-1}
=B_j^*\bigg(\Id+\sum_{j=1}^mB_jB_j^*\bigg)^{-1}.
\end{equation}
Moreover, let us denote 
by $x_{j,k}\in\XX_j$ the $j$th component of $x_{k}$,
by $u_{j,k}\in\XX_j$ the $j$th component of $u_{k}$, and
by $v_{j,k}\in\XX_j$ the $j$th component of $v_{k}$.
Furthermore, we denote by $\rho_{i,k}\in\RR$ the $i$th component of
$r_k$, by $\sigma_{i,k}\in\RR$ the $i$th component of $s_k$,
and by $\tau_{i,k}\in\RR$ the $i$th component of $t_k$. 
Then \eqref{e:9hfe6-03g} becomes
\begin{equation}
\label{e:9hfe6-03i}
\begin{array}{l}
\text{for}\;k=0,1,\ldots\\
\left\lfloor
\begin{array}{l}
q_k=\sum_{j=1}^mB_ju_{j,k}-s_k\\
\text{for}\;j=1,\ldots,m\\
\left\lfloor
\begin{array}{l}
x_{j,k}=u_{j,k}-R_jq_k\\
u_{j,k+1}=u_{j,k}+\mu_k\big(\prox_{\gamma\lambda\|\cdot\|_j}
(2x_{j,k}-u_{j,k})-x_{j,k}\big)\\
\end{array}
\right.\\[2mm]
r_{k}=\sum_{j=1}^mB_jx_{j,k}\\
\text{for}\;i=1,\ldots,n\\
\left\lfloor
\begin{array}{l}
\sigma_{i,k+1}=\sigma_{i,k}+\mu_k\big(\prox_{\gamma{\psi_i}}
(2\rho_{i,k}-\sigma_{i,k})-\rho_{i,k}\big).
\end{array}
\right.\\[2mm]
\end{array}
\right.\\[2mm]
\end{array}
\end{equation}
In large-scale problems, a possible drawback of this approach is 
that $m+n$ proximity operators must be evaluated at each iteration,
which can lead to impractical implementations in terms of
computations and/or memory requirements. The analysis of 
\cite[Corollary~5.5]{Siop15} shows that the proximity 
operators in \eqref{e:9hfe6-03i} can be sampled by sweeping
through the indices in $\{1,\ldots,m\}$ and $\{1,\ldots,n\}$
randomly while preserving the convergence of the iterates. 
This results in partial updates of the variables which lead to
significantly lighter iterations and remarkable flexibility in
the implementation of the algorithm. Thus, a variable 
$u_{j,k}$ is updated at iteration $k$ depending on whether a 
random activation variable $\varepsilon_{j,k}$ takes on the value
$1$ or $0$ (each component $\sigma_{i,k}$ of the vector $s_k$ 
is randomly updated according to the same strategy). 
The method resulting from this random sampling scheme is 
presented in the next theorem.

\begin{theorem}
\label{t:9hfe6-04}
Let $\mathsf{D}=\{0,1\}^{m+n}\smallsetminus\{\mathsf{0}\}$,
let $\gamma\in\RPP$, let $(\mu_k)_{k\in\NN}$ be a sequence in 
$\left]0,2\right[$ such that 
$\inf_{k\in\NN}\mu_k>0$ and $\sup_{k\in\NN}\mu_k<2$, 
let $(u_{j,0})_{1\leq j\leq m}\in\HH$,
let $s_0=(\sigma_{i,0})_{1\leq i\leq n}\in\RR^n$, and let 
$(\boldsymbol{\varepsilon}_{k})_{k\in\NN}=
(\varepsilon_{1,k}\ldots,\varepsilon_{m+n,k})_{k\in\NN}$ be 
identically distributed $\mathsf{D}$-valued random variables
such that, for every $i\in\{1,\ldots,m+n\}$,
$\operatorname{Prob}[\varepsilon_{i,0}=1]>0$. Iterate
\begin{equation}
\label{e:A}
\begin{array}{l}
\text{for}\;k=0,1,\ldots\\
\left\lfloor
\begin{array}{l}
q_k=\sum_{j=1}^mB_ju_{j,k}-s_k\\
\text{for}\;j=1,\ldots,m\\
\left\lfloor
\begin{array}{l}
x_{j,k}=u_{j,k}-R_jq_k\\[1mm]
u_{j,k+1}=u_{j,k}\,+\varepsilon_{j,k}\mu_k
\big(\prox_{\gamma\lambda \|\cdot\|_j}
(2x_{j,k}-u_{j,k})-x_{j,k}\big)\\[2mm]
\end{array}
\right.\\[2mm]
r_{k}=\sum_{j=1}^mB_jx_{j,k}\\
\text{for}\;i=1,\ldots,n\\
\left\lfloor
\begin{array}{l}
\sigma_{i,k+1}=\sigma_{i,k}\,+\varepsilon_{m+i,k}\mu_k
\big(\prox_{\gamma{\psi_i}}(2\rho_{i,k}-\sigma_{i,k})
-\rho_{i,k}\big).
\end{array}
\right.\\[2mm]
\end{array}
\right.
\end{array}
\end{equation}
Suppose that the random sequences 
$(\boldsymbol{\varepsilon}_{k})_{k\in\NN}$ and 
$(u_k,s_k)_{k\in\NN}$ are independent. Then, 
for every $j\in\{1,\ldots,m\}$, $(x_{j,k})_{k\in\NN}$
converges almost surely to a vector $x_j$ and 
$y=\sum_{j=1}^mL_jx_j$ is a solution to \eqref{e:jfY6Tr-31o}.
\end{theorem}
\begin{proof}
It follows from \cite[Corollary~5.5]{Siop15} that 
$(x_{1,k},\ldots,x_{m,k})_{k\in\NN}$ converges almost surely to 
a solution to \eqref{e:kalas}. In turn, $y$ solves 
\eqref{e:jfY6Tr-31o}.
\end{proof}

\begin{remark}
The operators $(R_j)_{1\leq j\leq m}$ of \eqref{e:Rj} are computed
off-line only once and they intervene in algorithm \eqref{e:A}
only via matrix-vector multiplications.
\end{remark}

\begin{remark}
\label{r:u7r}
It follows from the result of \cite{Siop15} that, under suitable
qualification conditions, the conclusions of
Theorem~\ref{t:9hfe6-04} remain true for a general choice of 
the functions $f_j\in\Gamma_0(\XX_j)$ instead of $\|\cdot\|_j$,
and when $L$ does not have full rank. 
This allows us to solve the more general versions of 
\eqref{e:jfY6Tr-31a} in which the regularizer is not a norm 
but a function of the form \eqref{e:general-primal}.
\end{remark}

\section{Numerical experiments}
\label{sec:exps}
In this section we present numerical experiments applying the
random sweeping stochastic block algorithm outlined in 
Section~\ref{sec:opt} 
to sparse problems in which the regularization penalty is a norm
fitting our framework, as described in
Assumption~\ref{a:milanjfY6Tr-22}.
The goal of these experiments is to show concrete applications of
the class of norms discussed in the paper and to illustrate the
behavior of the proposed random block-iterative proximal splitting
algorithm. Let us stress that these appear to be the first
numerical experiments on this kind of block-coordinate method for
completely nonsmooth optimization problems with converging sequence
of iterates.

The setting we consider is binary classification with the hinge
loss and a latent group lasso penalty \cite{Jacob2009-GL}. Each
data matrix $A \in \RR^{n \times d}$ is generated with i.i.d.
Gaussian entries and each row $a_i$ of $A$ is normalized to have
unit $\ell^2$ norm. Similarly, the true model vector $y\in\RR^d$
is sparse and its nonzero entries are generated randomly on the unit 
$\ell^2$ sphere. The $n$ observations are then obtained as
$\beta_i=\textrm{sign}(\scal{a_i}{y})$. To induce classification
errors, a randomly chosen subset
of the observations have their sign reversed, with the value of
the noise determining the size of the subset, expressed as a
percentage of the total observations. For the implementation of
the algorithm, we require the proximity operators of the
functions $\|\cdot\|_j$ and $\psi_i$ in Theorem~\ref{t:9hfe6-04}. In
this case, these are the $\ell^2$-norm and the hinge loss, which
have straightforward proximity operators \cite{Livre1}.

In large-scale applications it is not possible to activate all the
functions and all the blocks due to computing and memory
limitations. The random sweeping algorithm \eqref{e:A} allows us 
to activate only some of the blocks by toggling the activation 
variables $(\varepsilon_{j,k})_{1\leq j\leq m}$ and
$(\varepsilon_{m+i,k})_{1\leq i\leq n}$.
The vectors are updated only when the corresponding activation 
variable is equal to $1$; otherwise it is equal to $0$ and no
update takes place. In our experiments we always activate the
variables $(\varepsilon_{m+i,k})_{1\leq i\leq n}$ as they are
associated to the set of training points which
is small in sparsity regularization problems. On the other hand,
only a fraction $\alpha$ of the variables
$(\varepsilon_{j,k})_{1\leq j\leq m}$ are activated,
which is achieved by sampling, at each iteration $k$, a subset of
$\lceil m\alpha\rceil$ distinct indices in $\{1,\dots,m\}$. In
light of Theorem~\ref{t:9hfe6-04}, convergence of the iterates is
guaranteed for every $\alpha\in\rzeroun$. It is natural to ask to
what extent these partial updates slow down the algorithm with
respect to the hypothetical fully updated version in which
sufficient computational power and memory are available. 

To investigate this question, in our first experiment 
$A\in\RR^{1000\times 10000}$, the true model vector 
$y$ has sparsity $95\%$, and we apply a $25\%$ 
classification error rate. The relaxation parameters
$(\mu_k)_{k\in\NN}$ are all set to $1.99$, the proximal parameter
$\gamma$ is set to $0.01$, and the regularization parameter
$\lambda$ is set to $0.1$. As a stopping rule for the algorithm
we use
\begin{equation}
\frac{|x_{k+1}-x_k|_2}{|x_k|_2}\leq 10^{-6}.
\end{equation}
We employ the chain latent group lasso penalty, whereby
the groups define contiguous sequences of length 10, with an
overlap of length 3, and the number of groups is 1429. 
Table~\ref{tab:hinge-loss-iterations} presents the time, number 
of iterations, and number of iterations normalized by the activation 
rate for the hinge loss and latent group lasso penalty for values
of the activation parameter 
$\alpha$ in $\{0.1, 0.2, \ldots, 1.0\}$. The normalized iteration
numbers are obtained by multiplying the actual iteration number
by the activation rate $\alpha$ in order to fairly quantify the 
global computational effort. Indeed, scaling the iterations by
the activation rate allows for a fair comparison between regimes
since the computational load of the algorithm per iteration is
proportional to the number of activated blocks. 
We observe that, while the absolute number of
iterations naturally increases as the activation rate decreases, 
the normalized number of iterations is remarkably
stable across the different regimes.
Thus, in large scale problems in which memory
space and processing power are limited, the standard optimization
algorithm \eqref{e:9hfe6-03i} with full activation rate is 
not suitable, whereas our random sweeping procedure can be 
easily implemented. 
Interestingly, Table~1 indicates that the normalized number of
iterations is not affected.
Figure~\ref{fig:objectives-square-hinge} depicts (top) the 
objective values for the problem and (bottom) 
the distance to the limiting solution for 
various activation rates. 
We note that the paths are similar for all activation rates, 
and the convergence is similarly fast.
This reinforces our findings that partial activation of the
blocks does not lead to any deterioration in normalized 
performance. 
\begin{figure}
\begin{center}
\includegraphics[width=0.625\columnwidth]%
{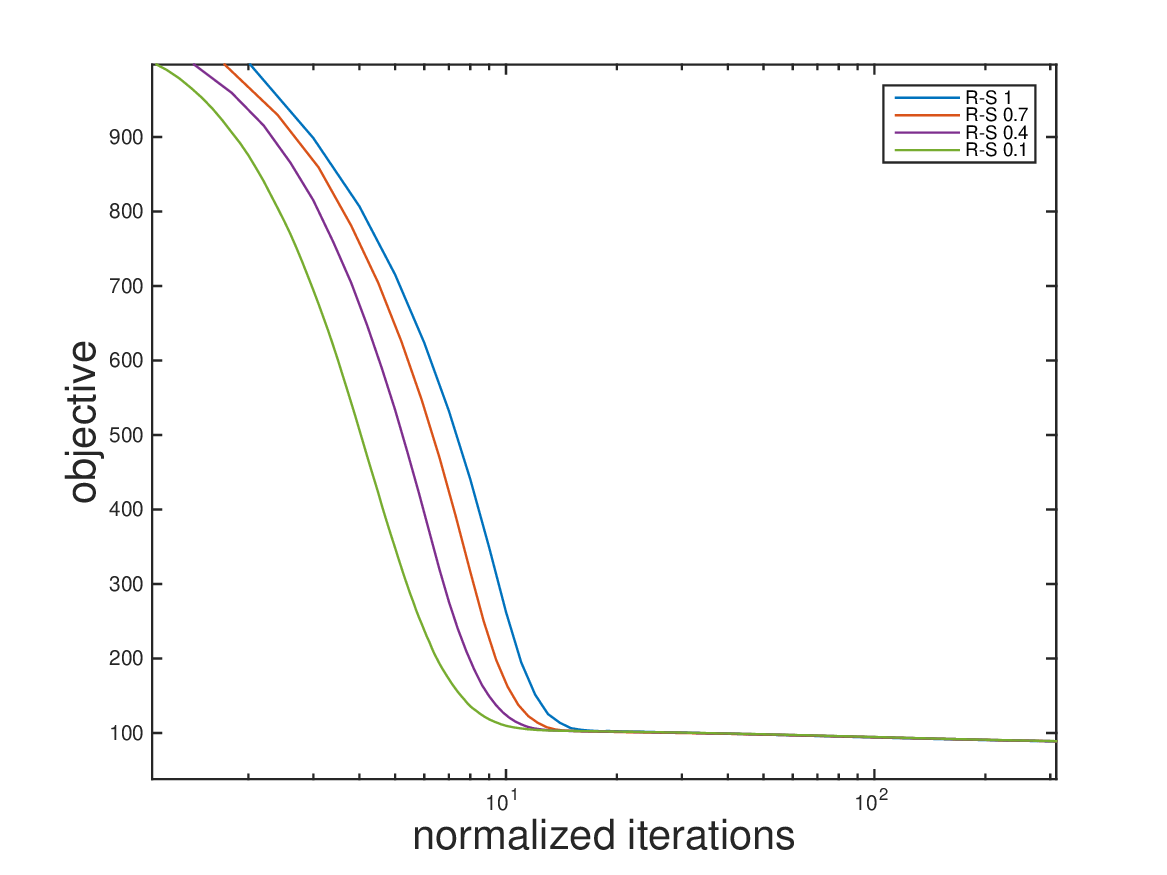}
\includegraphics[width=0.625\columnwidth]%
{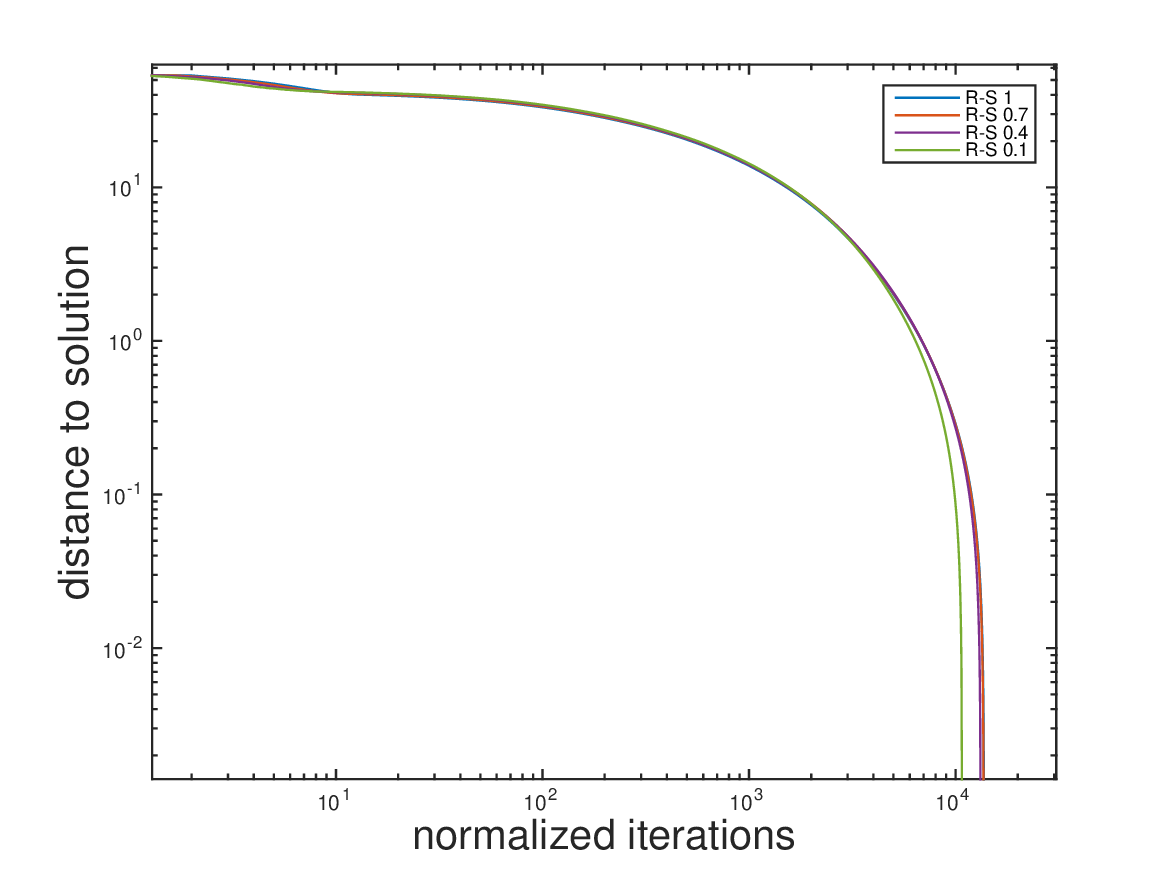}
\caption{Objective for hinge loss classification with the latent
group lasso (top), and distance to solution for the same
(bottom).}
\label{fig:objectives-square-hinge}
\end{center}
\end{figure}

As a second experiment, we revisit the above problem using
the $k$-support norm penalty of \cite{Argyriou2012}, which is a
special case of Example~\ref{ex:iit}. Here, 
$A\in\RR^{20\times 25}$, $k=4$, and $\alpha\in
\{0.005, 0.01, 0.05, 0.1, 0.2, 0.4, 0.6, 0.8, 1.0 \}$. 
The number of groups is $m=d!/(k!(d-k)!)=12650$.
Even in this relatively small size problem, the number 
of groups $m$ considerably exceeds both $d$ and $n$.
Table~\ref{tab:ksup-iterations} shows the same metrics as the
first experiment. We again observe that performance is
stable as $\alpha$ varies. 

\begin{table}[t]
\caption{Time, iterations, and normalized iterations 
for hinge loss classification with the latent group lasso. $A \in
\RR^{1000 \times 10000}$, $m=1429$.}
\vskip 0.3cm
\label{tab:hinge-loss-iterations}
\centering
\setlength\tabcolsep{4pt}
\begin{minipage}[th]{0.9\columnwidth}
\centering
\begin{small}
\begin{tabular}{lccc}
\toprule
activation & time (s) & actual & normalized  \\
 rate & & iterations & iterations \\ \midrule
1.0 & 24912 & 14515 & 14515 \\
0.9 & 24517 & 16047 & 14443 \\
0.8 & 26124 & 18080 & 14464 \\
0.7 & 28975 & 20633 & 14443 \\
0.6 & 28223 & 23872 & 14323 \\
0.5 & 29304 & 28308 & 14154 \\
0.4 & 36983 & 34392 & 13757 \\
0.3 & 38100 & 44080 & 13224 \\
0.2 & 50484 & 62664 & 12533 \\
0.1 & 62213 & 100829 & 10083 \\
\bottomrule
\end{tabular}
\end{small}
\end{minipage}
\end{table}

\begin{table}[t]
\caption{Time, iterations, and normalized iterations 
for hinge loss classification with the $k$-support norm. $A \in
\RR^{20 \times 25}$, $k=4$, $m=12650$.}
\vskip 0.3cm
\label{tab:ksup-iterations}
\centering
\setlength\tabcolsep{4pt}
\begin{minipage}[th]{0.9\columnwidth}
\centering
\begin{small}
\begin{tabular}{lccc}
\toprule
activation & time (s) & actual & normalized  \\
rate & & iterations & iterations \\
\midrule
1.0 & 388 & 463 & 463 \\
0.8 & 446 & 681 & 544 \\
0.6 & 454 & 894 & 536 \\
0.4 & 482 & 1281 & 512 \\
0.2 & 423 & 2557 & 511 \\
0.1 & 469 & 4851 & 485 \\
0.05 & 1054 & 9402 & 470 \\
0.010 & 1625 & 41633 & 416 \\
0.005 & 1907 & 77066 & 385 \\
\bottomrule
\end{tabular}
\end{small}
\end{minipage}
\end{table}

\end{document}